\documentclass[a4paper,reqno,12pt]{amsart}

\usepackage{titlesec}
\usepackage{amsmath,amsthm,amsfonts}
\usepackage{fullpage}
\usepackage{t1enc}
\usepackage{enumerate}
\usepackage{epsfig}
\usepackage{amssymb}
\usepackage{graphicx}
\usepackage{listings}
\usepackage{gensymb}
\usepackage{tikz}
\usepackage{tikz-cd}
\usepackage[utf8]{inputenc}
\usepackage[utf8]{luainputenc}
\usepackage[english]{babel}
\usepackage[top=1.5in,bottom=1.2in,left=1in,right=1in]{geometry}
\usepackage{mathtools}
\usepackage{pgf}
\usepackage[colorlinks=true, linkcolor=teal, citecolor=blue]{hyperref}
\usepackage{cleveref}
\usepackage{faktor}
\usepackage{xcolor}
\usepackage{xfrac}
\usepackage[all,cmtip]{xy}
\usepackage{xymtex}
\usepackage{xymtx-pdf}
\usepackage{MnSymbol}
\titleformat{\section}
  {\centering\normalfont\scshape\fontsize{12}{17}\bfseries}
  {\thesection}
  {1em}
  {}
\titleformat{\subsection}
  {\normalfont\fontsize{12}{14}\bfseries}
  {\thesubsection}
  {1em}
  {}

\newcommand{\D}[1]{\mathbb{D}^{#1}}

\newcommand{\Cp}[1]{\mathbb{C}\mathrm{P}^{#1}}
\newcommand{\Hp}[1]{\mathbb{H}\mathrm{P}^{#1}}
\newcommand{\Op}[1]{\mathbb{O}\mathrm{P}^{#1}}

\newcommand{\N}{\mathbb{N}}
\newcommand{\Z}{\mathbb{Z}}

\newcommand{\Sp}[1]{\mathbb{S}^{#1}}

\newcommand{\wi}[1]{\widetilde{#1}}
\newcommand{\smallsigma}{\scalebox{0.7}{\Sigma}}

\newcommand{\cc}[1]{\mathcal{C}{\left(#1\right)}}
 \newcommand{\csum}{{\#}}

\newcommand{\Th}[1]{\Theta_{#1}}

\newcommand{\dis}{\displaystyle}
\newcommand{\lr}{\longrightarrow}

\newcommand{\opl}[1]{\oplus{_{#1}}}
\newcommand{\kr}[1]{\mathrm{Ker}{\left(#1\right)}}
\newcommand{\ckr}[1]{\mathrm{Coker}{\left(#1\right)}}

\newcommand{\im}[1]{\mathrm{Im}{\left(#1\right)}}

\newcommand{\pii}[1]{\pi_{#1}}

\newtheorem{thm}{Theorem}[section]

\newtheorem{cor}[thm]{Corollary}
\newtheorem{propn}[thm]{Proposition}
\newtheorem{obs}[thm]{Observation}
\newtheorem{defn}{Definition}[section]
\newtheorem{rem}[thm]{Remark}

\numberwithin{equation}{section}

\Crefname{lemm}{Lemma}{Lemmas}
\crefname{lemm}{lemma}{lemmas}
\crefname{thm}{theorem}{theorems}
\Crefname{thm}{Theorem}{Theorems}

\title{Concordance structure set of connected sum of projective spaces}

\numberwithin{equation}{section}
\begin{document}

\begin{abstract}
In this paper, the concordance structure set of connected sums of complex and quaternionic projective spaces in the real $n$-dimensional range with $8\leq n\leq 16$ is computed. It is demonstrated that the concordance inertia group of a connected sum equals the sum of individual concordance inertia groups. Furthermore, the concordance structure sets of manifolds and their connected sums are compared.

\end{abstract}

% In this paper, a short exact sequence is established to compute the homotopy class of maps from a connected sum of manifolds into a homotopy commutative space. As an application, the concordance structure set of the $k$-fold connected sum of projective spaces is calculated. Moreover, it is demonstrated that the concordance inertia group of a connected sum equals the sum of individual concordance inertia groups.

%  \author{Samik Basu}
% \address{Stat-Math Unit, Indian Statistical Institute Kolkata, India, 700108.}
% \email{samik.basu2@gmail.com; samikbasu@isical.ac.in}
	
% \author{Ramesh Kasilingam}
	
% 	\address{Department of Mathematics,
% 		Indian Institute Of Technology Chennai, India, 600036.}
% \email{rameshkasilingam.iitb@gmail.com  ; rameshk@iitm.ac.in}

 \author{Priyanka Magar-Sawant}
	
	\address{Department of Mathematics,
		Indian Institute Of Technology Bombay, India, 400076.}
\email{priyanka.ms.math@gmail.com}

	

%%%%%%%%%%%%%%%%%%%%%%%%%%%%%%%%%%%%%%%%%%%%%%%%%%%%%%%%%%%%%%%%%%%%%%%%%%%%%%%%%%%%%%%%%%%%%%%%%%%%%%%%%%%%
\maketitle

%%%%%%%%%%%%%%%%%%%%%%%%%%%%%%%%%%%%%%%%%%%%%%%%%%%%%%%%%%%%%%%%%%%%%%%%%%%%%%%%%%%%%%%%%%%%%%%%%%%%%%%%%%%%%%%%%%%%%%%%%%%%%%%%%%%%%%%%%%%%%%%%%%%%%%%%%%%%%%%%%%%%%%%%%%%%%%%%%%%%%%%%%%%%%%%%

\section{Introduction}

In geometric topology, a vital problem, more or less equal to the classification problem, is the classification of smooth manifolds up to T-stabilization. The study of homotopy decomposition on the connected sum of a manifold with projective spaces is a typical way to stabilize the manifold, and studied in \cite{HomotopyofManifoldsStablizedbyProjectiveSpaces,Kasprowski2018FourmanifoldsUT,SurgeryAndDuality,kreck2015thoughts}. Particularly, the study of connected sum of projective spaces is essential. %In this paper, we study the connected sum of projective spaces up to concordance.
In this paper, we compute the concordance structure set of $k$-fold connected sum of complex projective spaces $\Cp{2n}$ with $l$-fold connected sum of quaternionic projective spaces $\Hp{n}$, where $k,l\in\N$ and $2\leq n\leq 4$. 
\begin{defn}\label{def: concr} Let $M$ be a closed smooth manifold. Let $(N, f )$ be a pair consisting of a smooth manifold $N$ together with a homeomorphism $f: N \to M $. Two such pairs $(N_1 , f_1 )$ and $(N_2 , f_2 )$ are concordant provided there exists a diffeomorphism $g : N_1 \to N_2$ and a homeomorphism
$F : N_1 \times [0, 1] \to M \times [0, 1]$ such that $F|_{N_1 \times0} = f_1$ and $F |_{N_1 \times 1} = f_2 \circ g$. 
\end{defn}
The concordance class of $(N, f )$ is denoted by $[N, f ]$, and the set of all such concordance classes is denoted by $\cc{M}$.
%The class $[M, Id]$ of the identity $\id: M \lr M$ can be considered as the base point of $\cc{M}$. 
Note that, from \cite[Part I]{HirschMazur}, concordance implies isotopy on smooth manifolds of dimensions $\geq 5$.

 In the smooth category of manifolds, the ultimate aim for a given manifold $M$ is to determine the set of all equivalence classes of smooth structures up to diffeomorphism, denoted by \textbf{$ {\mathcal{S}}(M) $} \cite{Cairns-Hirsch-Mazur}.
Acknowledging the inherent difficulty in directly addressing this problem, one first studies the concordance relation.
Observe that, there is a natural forgetful surjective map, 
		$ {\mathcal C} (M) \twoheadrightarrow {\mathcal{S}} (M)$.
Moreover, the group of self-homeomorphisms of $M$ acts on ${\mathcal C} (M)$, such that the orbit space of this action yields a bijection with the set ${\mathcal{S}}(M)$.
Another crucial reason for studying $\cc{M}$ is its correspondence in homotopy theory, as demonstrated by Kirby-Siebenmann that; if $ M $ admits a smooth structure, then there exists a set bijection 
$$\cc{M} \cong [M,\mathit{Top}/O],$$ 
 where $[M, \mathit{Top}/O]$ is in fact an abelian group \cite{KirbySiebenmannFoundationalEssays}. 
%=========================================================================================

In particular, this paper concludes the following:  
\begin{thm}
     The concordance structure sets $\cc{\csum_k\Cp{n}}$ for $4\leq n\leq 8$, $\cc{\csum_k\Cp{2n}\csum_l\Hp{n}}$ for $2\leq n\leq 4$, and $\cc{\csum_k\Cp{8}\csum_l\Hp{4}\csum_m\Op{2}}$, where $k,l,m\in \N$, are determined. (see \Cref{sec 3.1}
\end{thm}
%==================================================================================================================================================

Recall that, for the case of standard $n$-sphere, the set $\mathcal{S}(\Sp{n})$ coincides with $\cc{\Sp{n}}$, which is also equivalent to the group $\Theta_n$ of $h$-cobordism classes of homotopy $n$-spheres, for all $n\geq 5$ \cite{KervaireMilnorGroupsOfHomotopySpheres}. 
In the classification problem, the group $\Theta_n$ plays an important role. Specifically, if we take a connected sum of a smooth manifold with an exotic sphere, the resulting manifold remains homeomorphic to the underlying topological manifold; however, it might change the diffeomorphism class, resulting in an exotic structure on the manifold. Therefore, determining such elements of $\Theta_n$ correlate with the classification problem. Certain important subgroups of $\Th{n}$ studied in this paper are defined below:
\begin{defn} Let $M$ be a closed  oriented smooth $n$-manifold.
\begin{enumerate}[(i)]
    \item  The inertia group of $M$ is the subgroup $I(M )\subseteq \Th{n} $ of homotopy spheres $\Sigma^n$ such that $M\csum\Sigma^n$ is diffeomorphic to $M$.

    \item The concordance inertia group $I_c(M)$ consists of $\Sigma^n\in I(M)$ such that there exists a diffeomorphism from $M\csum\Sigma^n\to M$ that is concordant to the canonical homeomorphism $h_{\smallsigma^n}:M\csum\Sigma^n\to M$.
\end{enumerate}
\end{defn}
 In  \cite{WilkensInertiaGroupSumIsNotSumOfInertia}, Wilkens demonstrated that the equality $I(M\csum N)=I(M)\oplus I(N)$ does not hold universally. In contrast to this finding, our paper demonstrates that this equality holds for the concordance inertia group. In \Cref{ic of connected sum is sum of I_c}, we see that, the concordance inertia group of individual manifolds holds the complete information needed to determine the concordance inertia group of their connected sum.
\begin{thm}
    The concordance inertia group of a connected sum of manifolds, is equal to the sum of the concordance inertia groups of individual manifolds.
\end{thm}

%=======================================================================================================================================================================
 In the category of manifolds, the classification problem involves studying different sets (see \cite{KirbySiebenmannFoundationalEssays,kui,Cairns-Hirsch-Mazur,mun,Sullivan1996}), and interestingly, many of these sets like stable cohomotopy groups, real/complex $K$-groups, (tangential) normal invariant sets, etc., exhibit a bijective correspondence with  $[M, Y]$, for a manifold $M$ and an $H$-space $Y$. It is worth noting that, the set $[M,Y]$ becomes more accessible from the viewpoint of algebraic topology as it fits into an exact sequence of abelian groups.
 In this paper, spaces of our interest are  the connected sum of closed oriented smooth $n$-manifolds $M_i$, denoted by $\#_{i\in I}M_i$, for $i\in I=\{1,2,\dots,k\}$. \Cref{thm:gen short exact sequence} presents a short exact sequence for computing the group $[\#_{i\in I}M_i,Y]$ in terms of values associated with $[M_i,Y]$. Additionally, we establish conditions on the collapse map $\psi:M_1\#M_2\to M_1$ to clarify the relation between the groups $[M_1,Y]$ and $[M_1\#M_2,Y]$. 
 In special case, we observe that,
 
 \begin{thm}
     The induced map $ \psi^*:\cc{M_1}\lr \cc{M_1\#M_2}$ is injective if and only if $ I_c(M_1\#M_2)=I_c(M_1) $.
 \end{thm} 
 %Here, $I_c(M)$ is the concordance inertia group of $M$. % Below, definitions and the importance of these groups are provided, with corresponding computations presented in Section \ref{sec 5}.\\
 
%%%%%%%%%%%%%%%%%%%%%%%%

  \noindent\textbf{Convention:} We work in the category of manifolds, which are  connected closed oriented smooth of dimension $\geq 5$. Furthermore, all morphisms in this category are considered to be orientation-preserving smooth.
 The notations $I$, $d$ and $p$ are reserved for the indexing set $I=\{1,2,\dots,k\}$, the degree one map $d:M^n\to \Sp{n}$ and  the pinch map $p:\Sp{n}\to \vee_{i\in I}\Sp{n}_i$, respectively, for $n\in \N$.\\

\noindent\textbf{Organization of the paper:}
In Section \ref{sec 4}, we present a short exact sequence containing the group $[\csum_{i\in I}M_i^n,Y]$ for homotopy-commutative $H$-groups $Y$. We conclude the section by exhibiting a relation between the groups $[M_1,Y]$ and $[M_1\csum M_2,Y]$. In Section \ref{sec 5}, it is demonstrated that the concordance inertia group of a connected sum is the sum of the individual concordance inertia groups. As an application of previous results, explicit values of the $\cc{\csum_k\mathbb{P}^n}$, where $\mathbb{P}^n=\Cp{n}, \Hp{n}$ or $\Op{n}$ with real-dimension $\leq16$.

%%%%%%%%%%%%%%%%%%%%%%%
\section{Maps from the connected sum of manifolds to homotopy commutative \texorpdfstring{$H$}{H}-groups}\label{sec 4}

Let $ M_1 $, $ M_2 $, $ \dots $, $ M_k $ be closed oriented smooth $ n $-dimensional manifolds with minimal cellular decomposition $ M_i\simeq M_i^{(n-1)}\bigcup_{h_i}\D{n} $, where $ h_i:\Sp{n-1}_i\to M_i^{(n-1)} $ is the attaching map for the top cell with $ i\in I=\{1,2,\dots,k\} $. Let $ \csum_{i\in I}M_i $ denote the connected sum of $ M_1 , M_2 , \dots , M_k $. Then we have 
$$\underset{i\in I}{\csum}M_i\simeq\left(\underset{i\in I}{\vee}M_i^{(n-1)}\right)\underset{\xi}{\bigcup}~\D{n},$$
where $ \xi=(\vee_{i\in I}h_i)\circ p$ is the attaching map for the top cell with $ p:\Sp{n-1}\to \vee_{i\in I}\Sp{n-1}_i $ as the pinch map. Thus we have the cofiber sequence $\Sp{n-1}\overset{\xi}{\to} \vee_{i\in I}M_i^{(n-1)} \overset{\iota}{\hookrightarrow}\csum_{i\in I}M_i  $. 
From this sequence, we obtain the following long exact sequence by applying the contra-variant functor $ [\_,Y] $, for a homotopy-commutative $H$-group $ Y $,  

\begin{equation}\label{main seq}
\begin{tikzcd}[column sep=2.2em]
	{\cdots\underset{i\in I}{\opl{}}[\Sigma^{s+1}M_i^{(n-1)},Y]} & {[\Sp{n+s},Y]} & {[\Sigma^{s}(\underset{i\in I}{\csum}M_i),Y]} & {\dis\underset{i\in I}{\opl{}}[\Sigma^s M_i^{(n-1)},Y]\dots} \\
	{[\Sp{n},Y]} & {[\underset{i\in I}{\csum}M_i,Y]} & {\dis\underset{i\in I}{\opl{}}[M_i^{(n-1)},Y]} & {[\Sp{n-1},Y],}
	\arrow["{(\Sigma^{s+1}\xi)^*}", from=1-1, to=1-2]
	\arrow["{(\Sigma^{s}d)^*}", from=1-2, to=1-3]
	\arrow["{(\Sigma\xi)^*}"{description},overlay, in=168, out=-10, from=1-4, to=2-1]
	\arrow["{d^*}", from=2-1, to=2-2]
	\arrow["{\iota^*}", from=2-2, to=2-3]
	\arrow["{\xi^*}", from=2-3, to=2-4]
	\arrow["{(\Sigma^{s}\iota)^*}", from=1-3, to=1-4]
\end{tikzcd}
\end{equation}
where $ d^* $ is induced by the degree one map $ d:\csum_{i\in I}M_i\to \Sp{n} $ and $ \Sigma^s $ denotes the positive $ s $-fold suspension. \\

We utilize the property of the pinch map in sequence \eqref{main seq} to derive the following theorem:
\begin{thm}{\label{thm:gen short exact sequence}}
Let $ Y $ be a homotopy-commutative $ H $-group. Then there exists a short exact sequence
\[\begin{tikzcd}
	0 & {\ckr{(\Sigma^{s+1}\xi)^*}} & {[\Sigma^{s}(\csum_{i\in I}M_i),Y]} & {\kr{(\Sigma^{s}\xi)^*}} & 0,
	\arrow[from=1-1, to=1-2]
	\arrow["{\wi{(\Sigma^sd)^*}}", from=1-2, to=1-3]
	\arrow["{(\Sigma^{s}\iota)^*}", from=1-3, to=1-4]
	\arrow[from=1-4, to=1-5]
\end{tikzcd}\]
where the map $ \wi{(\Sigma^sd)^*} $ is induced by the degree one map $\Sigma^sd :\Sigma^{s}(\csum_{i\in I}M_i)\to \Sp{n+s}  $, 
$$\ckr{(\Sigma^{s+1}\xi)^*}= \faktor{\pii{n+s}(Y)}{\sum_{i\in I}\im{(\Sigma^{s+1} h_i)^*}}$$
and $$\kr{(\Sigma^s\xi)^*}=\kr{{\underset{i\in I}\sum}(\Sigma^sh_i)^*}~.$$
\end{thm}
 \begin{proof}
Using the long exact sequence \eqref{main seq}, we acquire the following short exact sequence
\[\begin{tikzcd}
	0 & {\faktor{\pii{n+s}(Y)}{\kr{(\Sigma^sd)^*}}} & {[\Sigma^s(\csum_{i\in I}M_i),Y]} & {\im{(\Sigma^s\iota)^*}} & 0,
	\arrow[from=1-1, to=1-2]
	\arrow["{\wi{(\Sigma^sd)^*}}", from=1-2, to=1-3]
	\arrow["{(\Sigma^s\iota)^*}", from=1-3, to=1-4]
	\arrow[from=1-4, to=1-5]
\end{tikzcd}\]
where $ \wi{(\Sigma^sd)^*} $ is the induced map by the degree one map $ (\Sigma^sd) :\Sigma^{s}(\csum_{i\in I}M_i)\to \Sp{n+s}$. Using exactness of the sequence \eqref{main seq}, we have $$ \im{(\Sigma^s\iota)^*}=\kr{(\Sigma^s\xi)^*} $$ and $$ \kr{(\Sigma^sd)^*}=\im{(\Sigma^{s+1} \xi)^*}. $$ This offers the existence of the required short exact sequence.

 Now we express the $\im{(\Sigma^{s+1} \xi)^*}$ in terms of
$\im{(\Sigma^{s+1}h_i)^*}$, and the $\kr{(\Sigma^{s} \xi)^*}$ in terms of $(\Sigma^{s} h_i)^*$. Note that $ \xi=(\vee_{i\in I}h_i)\circ p $, where $ p $ is the pinch map. For any $ m\in \N $, the pinch map $ p:\Sp{m}\to\bigvee_{i\in I}\Sp{m}_i  $ induces the map $p^*:\opl{i\in I}[\Sp{m}_i,Y]\to[\Sp{m},Y]  $  with the property that it 
maps any element $ (x_1,x_2,\dots,x_k) $ to $ \sum_{i\in I}x_i $. This implies that for any positive $ s $-fold suspension of $ \xi $, we have $ (\Sigma^s\xi)^*=\sum_{i\in I}(\Sigma^sh_i)^*$ which leads to the following:
\begin{equation}\label{eq:im of xi is equal to the sum of hi}
\im{(\Sigma^{s+1}\xi)^*}=\sum_{i\in I}\im{(\Sigma ^{s+1}h_i)^*} 
\end{equation} 
and
\begin{equation}\label{eq:ker of xi is ker of sum of hi}
 \kr{(\Sigma^{s}\xi)^*}=\kr{\sum_{i\in I}(\Sigma^{s}h_i)^*} ~.
\end{equation}
This completes the proof.
\end{proof}
If we consider a $k$-fold connected sum of manifold $M$, $\csum_kM$, in the \Cref{thm:gen short exact sequence}, then we have the following version:
\begin{thm}\label{main thm for connected sum of same copy}
Let $ M $ be a closed oriented smooth $ n $-manifold. Then the
$$\ckr{(\Sigma^{s+1}\xi)^*}=  \faktor{\pii{n+s}(Y)}{\im{(\Sigma^{s+1} h)^*}}$$ 
and $$\kr{(\Sigma^{s}\xi)^*}\cong  \big(\opl{k-1}[\Sigma^{s}M^{(n-1)},Y]\big)\oplus\kr{(\Sigma^{s}h)^*},$$\\
\noindent where $ \xi=(\vee_kh)\circ p $ is the attaching map for $ \csum_kM $ with $ h:\Sp{n-1}\to M^{(n-1)} $ as the attaching map for $ M $ and $ p $ is the pinch map.
\end{thm}
\begin{proof}
We have $ \xi=(\vee_kh)\circ p $, hence as a particular case of \eqref{eq:im of xi is equal to the sum of hi}, we obtain $$ \im{(\Sigma^{s+1}\xi)^*} = \im{(\Sigma^{s+1}h)^*} .$$ 

Moreover, we have
\begin{center}
$\kr{(\Sigma^{s}\xi)^*}\cong (\opl{k-1}[\Sigma^{s}M^{(n-1)},Y])\oplus\kr{(\Sigma^{s}h)^*}.$ 
\end{center} 
This holds since any element $ (x_1,x_2,\dots,x_k)$ of $\opl{k}[\Sigma^{s}M^{(n-1)},Y] $ lies in $ \kr{(\Sigma^{s}\xi)^*} $ if and only if $ \sum_{i\in I}(\Sigma^s h)^*(x_i)=(\Sigma^sh)^*(\sum_{i\in I}x_i) $ is $ 0 $. Therefore the sum $ \sum_{i\in I}x_i=x $ for some $ x\in\kr{(\Sigma^{s}\xi)^*} $ which implies $ x_k=x-(\sum_{i=1}^{k-1}x_i) $. Hence the map  
\begin{align*}
    (\opl{k-1}[\Sigma^{s}M^{(n-1)},Y])\oplus\kr{\Sigma^{s}h^*} &\longrightarrow \kr{(\Sigma^{s}\xi)^*} \\
    (x_1,x_2,\dots,x_{k-1},x) &\mapsto \left(x_1,x_2,\dots,x_{k-1},x -(\textstyle \sum_{i=1}^{k-1}x_i) \right)
\end{align*}
is indeed an isomorphism, and this completes the proof.
\end{proof}

%%%%%%%%%%%%%%%%%%%%%%%%%%%%%%%%%%%%%%%%%%%%%%%%%%%%%%%%%%%%
%%%%%%%%%%%%%%%%%%%%%%%%%%%%%%%%%%%%%%%%%%
We conclude this section by exhibiting a relation between the groups $[M_1,Y]$ and $[M_1\csum M_2,Y]$. In particular, we present conditions for the injectivity and surjectivity of the induced collapse map from $[M_1,Y]$ to $[M_1\csum M_2,Y]$.
%%%%%%%%%%%%%%%%%%%%%%%%%%%%%%%%%%%%%%%%%%
\begin{thm}\label{M to M connected sum N}
Let $ M_1$ and $M_2 $ be closed oriented smooth $ n $-manifolds with $ M_i=M_i^{(n-1)}\bigcup_{h_i}\D{n} $, where $ h_i:\Sp{n-1}\to M_i^{(n-1)}  $ is the attaching map for $ i=1,2 $. Let $ \psi:M_1\csum M_2\to M_1 $ be the collapse map and $ Y $ be a homotopy-commutative $ H$-group. Then for the attaching map $ \xi:=(h_1\vee h_2) \circ p$ of $ M_1\csum M_2$, where $ p $ is the pinch map, the following holds:
\begin{enumerate}[(i)]
\item The map $ \psi^*:[M_1,Y]\to [M_1\csum M_2,Y] $ is injective if and only if $ \im{(\Sigma\xi)^*}=\im{(\Sigma h_1)^*} $.%$ I_c(M_1\csumM_2)=I_c(M_1) $.

\item The map $ \psi^*:[M_1,Y]\to [M_1\csum M_2,Y] $ is surjective if and only if $ \kr{\xi^*}\cong\kr{h_1^*} $.

\end{enumerate}
\end{thm}
\begin{proof}
Consider the following homotopy commutative diagram
\[\begin{tikzcd}
	{\Sp{n-1}} & {M_1^{(n-1)}\vee M_2^{(n-1)}} & {M_1\csum M_2} & \cdots \\
	{\Sp{n-1}} & {M_1^{(n-1)}} & {M_1} & \cdots
	\arrow["{h_1}"', from=2-1, to=2-2]
	\arrow["{\iota_1}"', hook, from=2-2, to=2-3]
	\arrow["d"', from=2-3, to=2-4]
	\arrow["\xi", from=1-1, to=1-2]
	\arrow["g", from=1-1, to=2-1]
	\arrow["\phi", from=1-2, to=2-2]
	\arrow["\iota", hook, from=1-2, to=1-3]
	\arrow["\psi", from=1-3, to=2-3]
	\arrow["d", from=1-3, to=1-4]
\end{tikzcd}\]
 where $ \xi=(h_1\vee h_2) \circ p$ for the pinch map $ p $, horizontal sequences are cofiber sequences associated to maps $ \xi $ and $ h_1 $, and maps $ g,\, \phi$ and $\psi  $ are collapse maps. Now, applying the contra-variant functor $ [\_,Y] $ onto this diagram, we acquire the following commutative diagram 
\begin{equation}\label{collapse map comm dig for shoer ex seq}
    \begin{tikzcd}
	0 & {\sfrac{\pii{n}(Y)}{\im{(\Sigma h_1)^*}}} & {{[M_1,Y]}} & {\kr{h_1^*}} & 0 \\
	0 & {\sfrac{\pii{n}(Y)}{(\im{(\Sigma \xi)^*}}}& {[M_1\csum M_2,Y]} & {\kr{\xi^*}} & 0
	\arrow[from=1-1, to=1-2]
	\arrow["{\wi{d^*}}", from=1-2, to=1-3]
	\arrow["\iota_1^*", from=1-3, to=1-4]
	\arrow[from=1-4, to=1-5]
	\arrow["\wi{(\Sigma g)^*}", from=1-2, to=2-2]
	\arrow[from=2-1, to=2-2]
	\arrow["{\wi{d^*}}", from=2-2, to=2-3]
	\arrow["{\psi^*}", from=1-3, to=2-3]
	\arrow["\iota^*", from=2-3, to=2-4]
	\arrow["\wi{\phi^*}", from=1-4, to=2-4]
	\arrow[from=2-4, to=2-5]
\end{tikzcd}
\end{equation}
where horizontal short exact sequences are obtained using \Cref{thm:gen short exact sequence}. Observe that, the induced map $\wi{(\Sigma g)^*}:\sfrac{\pi_{n}(Y)}{\im{(\Sigma h_1)^*}}\to\sfrac{\pi_{n}(Y)}{\im{(\Sigma \xi)^*}}  $, by $ (\Sigma g)^* $, is surjective; and the induced map $ \wi{\phi^*}:\kr{h_1^*}\to \kr{\xi^*} $, by $ \phi^* $, is injective.

\begin{enumerate}[$ (i) $]
\item  Let $ \psi^*:[M_1,Y]\to [M_1\csum M_2,Y] $ be injective. This implies $ \wi{(\Sigma g)^*} $ is an isomorphism. Note that we always have $ \im{(\Sigma h_1)^*}\subseteq \im{(\Sigma\xi)^*} $. Therefore, we obtain $ \im{(\Sigma h_1)^*}=\im{(\Sigma\xi)^*} $.

On the other hand, $ \im{(\Sigma h_1)^*}= \im{(\Sigma\xi)^*} $ implies the map $  \wi{(\Sigma g)^*}  $ is injective. Using the fact that map $ \wi{\phi^*} $ is injective in \eqref{collapse map comm dig for shoer ex seq}, by five lemma we acquire the result.

\item Suppose the map $ \psi^*:[M_1,Y]\to [M_1\csum M_2,Y] $ is surjective. Therefore, in \eqref{collapse map comm dig for shoer ex seq}, the map $ \wi{\phi^*} $ turns out to be surjective, and hence an isomorphism.

Conversely, suppose the map $ \wi{\phi^*} $ be an isomorphism. Then, using the fact that $  \wi{(\Sigma g)^*} $ is surjective, by five lemma we conclude that the map $ \psi^*  $ is surjective, which completes the proof.
  \end{enumerate}

 \end{proof}
Recall that, the kernel of the induced degree one map map $d^*:[\Sp{n}, Top/O]\to[M, Top/O]$ can be identified with the concordance inertia group of $M$, $I_c(M)$. In the \Cref{M to M connected sum N} $(i)$, if we consider $Y = Top/O$, then the equivalent statement we have is as follows:
\begin{cor}
\label{si is injective iff ic(m1+m2)=ic(m1)}
The map $ \psi^*:\cc{M_1}\to \cc{M_1\csum M_2}$ is injective if and only if $ I_c(M_1\csum M_2)=I_c(M_1) $.
    \end{cor}

%In the upcoming section, we explore several computational corollaries of Theorems discussed thus far in this section. %These corollaries provide further insights and expand upon the results introduced in this section.

%%%%%%%%%%%%%%%%%%%%%%%%%%%%%%%%%%%%%%%%%%%%%%%%%%%%%%%%%%%%%%%%%%%%%%%%%%%%%%%%%%%%
%%%%%%%%%%%%%%%%%%%%%%
\section{Sum formula and the concordance structure set of projective spaces}\label{sec 5}

In this section, we begin with establishing a crucial result, stating that the concordance inertia group of a connected sum of manifolds, is equal to the sum of the concordance inertia groups of individual manifolds. In other words, the concordance inertia groups of individual manifolds hold the complete information needed to determine the concordance inertia group of their connected sum.

%%%%

Let $ M_i $ be a closed oriented smooth $ n $-manifold with the attaching map $ h_i:\Sp{n-1}\to M_i^{(n-1)} $, for $ i\in I $. Recall that, if $d: M_i\to \Sp{n}$ is a degree one map, then the kernel of the induced map $d^*:[\Sp{n},Top/O]\to [M_i,Top/O]$ can be identified with the concordance inertia group of $ M_i $, $I_c(M_i)$. The induced long exact sequence associated with the map $ h_i $,
\begin{tikzcd}
	{\cdots[\Sigma M_i^{n-1},Top/O]} & {[M_i,Top/O]} & {[\Sp{n},Top/O]}
	\arrow["{(\smallsigma h_i)^*}", from=1-1, to=1-2]
	\arrow["{d^*}", from=1-2, to=1-3]
\end{tikzcd}, gives $$ I_c(M_i)=\kr{d^*}=\im{(\Sigma h_i)^*} .$$ Now, let us consider $ Y=Top/O $ in the sequence \eqref{main seq} then we have $ \kr{d^*}=\im{(\Sigma\xi)^*} $. Combining this with the equation \eqref{eq:im of xi is equal to the sum of hi}, we have the following theorem.

%%%%%%%%%%%%%%%%%%%%%%%%%%%%%%%%%%%%%%%%%%%%%%%%%%%%%%%%%%%%%%%%%%%%%%%%
\begin{thm}\label{ic of connected sum is sum of I_c}
Let $ M_1,M_2,\dots,M_k $ be closed oriented smooth $ n $-dimensional manifolds. Then 
    $$ I_c\left({\csum_{i\in I}} M_i\right)=\underset{i\in I}{\sum}I_c(M_i)~.$$
In particular, if $ M_i=M $, for all $ i=1,2,\dots,k $, then 
$$I_c\left({\csum_k} M\right)=I_c(M) ~ .$$

\end{thm}
\begin{obs}
The group $I_c(\csum_{i\in I}M_i)$ is trivial if and only if each $I_c(M_i)$ is trivial for $i\in I$. 
\end{obs}
% To further emphasize this observation, we present the following corollary in conjunction with the \Cref{ic of 8910}:
% \begin{cor}
% Let $ M^n_1,M^n_2,\dots,M^n_k $ be closed oriented smooth $ n $-dimensional manifolds. Then 
%    \begin{enumerate}[(i)]
%        \item $I_c(\csum_{i=1}^kM^8_i)=0$.
     
%     \item  $I_c(\csum_{i=1}^kM^9_i)=\Z/2$ if and only if at least one manifold $M^9_i$ is non-spin.

%       \item Suppose each $M^{10}_i$ is simply-connected. Then $I_c(\csum_{i=1}^kM^{10}_i)=\Z/2$ if and only if at least one manifold $M^{10}_i$ is non-spin.
%        \end{enumerate}
  
% \end{cor}
%%%%

Now we will illustrate the \Cref{si is injective iff ic(m1+m2)=ic(m1)}.
Let $ A $ be $ (n-1) $-connected closed oriented smooth $ 2n$-manifold. According to Wall's theorem, $A$ can be represented as the minimal cell decomposition $ A=\left(\vee_r\Sp{n}\right)\cup_h\D{2n} $, where $ h:\Sp{2n-1}\to\vee_r\Sp{n} $ is the attaching map. Within the range of $3\leq n\leq 6$, by utilizing the cofiber sequence derived from the attaching map $h$ and the fact that $\Th{m}=0$  for $m=3,4,5,6,12$, it becomes apparent that $I_c(A)=0$. Thus, because of the \Cref{ic of connected sum is sum of I_c}, for any closed oriented smooth $ 2n $-dimensional manifold $M$ with $3\leq n \leq 6$ we get $I_c(M\#A)=I_c(M)$, and the \Cref{si is injective iff ic(m1+m2)=ic(m1)} says that the map $ \psi^*:\cc{M}\to \cc{M\csum A}$ is injective. In these cases, we also get that the map $\psi^*$ is surjective. This can be easily verified with the help of the cofiber sequence $\vee_r\Sp{n}\overset{\iota}{\hookrightarrow} {M\#A} \overset{q}{\to} M$, where $\iota$ is the inclusion of $ (2n-1) $-skeleton of $ A $. As an outcome, we get the following result.
%%%%%%%%%%%%%%%%%
\begin{cor}\label{cor: C of n-1 connected 2n mfld}
Let $ M$ and $ A $ be closed oriented smooth $ 2n $-dimensional manifolds. In addition, if $ A $ is $ (n-1) $-connected manifold then, for $ 3\leq n \leq 6$, $$\cc{M\csum A}\cong\cc{M}~.$$
\end{cor}
For example, if $ M^{2n} $ is a closed oriented smooth manifold then we have the following results for any $ k\in \N $:
\begin{enumerate}[$(i)$]
\item  $ \cc{M^8\csum(\csum_k\Hp{2})} \cong \cc{M^8}$.

\item 	$ \cc{M^{10}\csum(\csum_k(\Sp{5}\times\Sp{5}))}\cong\cc{M^{10}}$.

\item   $ \cc{M^{12}\csum(\csum_k(\Sp{6}\times\Sp{6}))}\cong\cc{M^{12}} $.

\end{enumerate}

%%%%

Recall that, the \emph{homotopy inertia group} $I_h(M^n)$ consists of $\Sigma^n\in \Th{n}$ such that there exists a diffeomorphism from $M^n\csum\Sigma^n\to M^n$ that is homotopic to the canonical homeomorphism $h_{\smallsigma^n}:M^n\csum\Sigma^n\to M^n$. The sum formula is also valid for the homotopy inertia group under certain conditions on manifolds, stated below.
\begin{propn}\label{ih of connected sum}
Let $ M_1,M_2,\dots,M_k $ be closed simply-connected smooth $ n $-dimensional manifolds such that for each $ M_i $ odd degree cohomologies with integral and mod-2 coefficients vanish. Then
$$I_h\left(\underset{i\in I}{\csum}M_i\right)
=\underset{i\in I}{\sum}I_h(M_i)~.$$ 
\end{propn}
\begin{proof}
 By \cite[Corollary3.2]{RKHomotopyInertiaGroupsAndTangentialStructures}, if $M_i$ is a closed simply-connected smooth manifold such that odd degree cohomologies with integral and mod-2 coefficients vanish, then $I_h(M_i)=I_c(M_i)$.
Thus, assuming conditions from the statement we obtain ${\sum_{i\in I}}I_h(M_i)={\sum_{i\in I}}I_c(M_i)$ and $I_h({\csum_{i\in I}}M_i)=I_c({\csum_{i\in I}}M_i)$. Combining this with the \Cref{ic of connected sum is sum of I_c} completes the proof.
\end{proof}
\begin{rem}\label{ic of connected sum of projective spaces is trivial} Let $k,n\in \N$.
     \begin{enumerate}[(i)]
    \item If $n\leq 8$ then $I_h(\csum_k\Cp{n})=0$, since $I_c(\Cp{n})=0$ by \cite{Kawakubo}.
    
    \item If $n\leq 4$ then $I_h(\csum_k\Hp{n})=0$, since $I_c(\Hp{n})=0$ by \cite{AravindaExoticStructuresAndQuaternionicHyperbolicManifolds}.
    
    \item   By \cite{RKInertiaGroupsAndSmoothStructuresOf2nmfld} $I_c(\Op{2})=0$ hence $I_h(\csum_k\Op{2})=0$.
\end{enumerate}
 \end{rem}
In the upcoming subsection, we explore several computational corollaries of Theorems discussed thus far.
%---------------------------------------------------------------------------------------------

\subsection{Concordance structure set of \texorpdfstring{$k$}{k}-fold connected sum of projective spaces}\label{sec 3.1}

Now, we proceed to calculate the concordance structure set for $\csum_k\mathbb{P}^n$, where $\mathbb{P}^n$ denotes either $\Cp{n}$, $\Hp{n}$, or $\Op{2}$ with corresponding real dimensions $2n$, $4n$, or $16$ respectively.

%---------------------------------------------------------------------------------------------
%%%%%%%%%%%%%%%%%

 Now, let us recall that $ \Cp{n}=\Cp{n-1}\cup_{h_C}\D{2n} $, where $ h_C:\Sp{2n-1}\to\Cp{n-1} $ is the Hopf map. The groups $\cc{\Cp{n} } $ are computed for $n=3,4$ in \cite{RK16ClassificationSmoothStructuresComplexProjectiveSpace} and for $5 \leq n \leq 8$ in \cite{RK-2017-CP5-8}. Using \Cref{thm:gen short exact sequence,,main thm for connected sum of same copy}, now we compute the concordance structure set of $k$-fold connected sum of complex projective spaces.

\begin{cor}\label{cor:sq for cpn}
\begin{enumerate}[(i)]
\item  There is an isomorphism $$ d^*:\Theta_{8}\lr\cc{\csum_k\Cp{4}}, $$
		where $ \Theta_{8}\cong\Z_2 $.

\item There is a short exact sequence
		\[\begin{tikzcd}\label{c(cp5)}
			0 & {\Theta_{10}} & {\cc{\csum_k\Cp{5}}} & {\opl{k}\cc{\Cp{4}}} & 0,
			\arrow[from=1-1, to=1-2]
			\arrow["{{d^*}}", from=1-2, to=1-3]
			\arrow["{{\iota^*}}", from=1-3, to=1-4]
			\arrow[from=1-4, to=1-5]
		\end{tikzcd}\]
		where $  \Theta_{10}\cong\Z_2\opl{}\Z_3 $ and $  \cc{\Cp{4}}\cong\Z_2$.

\item There is an isomorphism 
	$$\iota^*:\cc{\csum_k\Cp{6}} \lr\opl{k-1}\cc{\Cp{5}}\oplus\kr{h_C^*},$$
	where $ \cc{\Cp{5}}\cong \Z_2^2\opl{}\Z_3$ and $ \kr{h_C^*:\cc{\Cp{5}}\lr\Th{11}}\cong\Z_2\oplus\Z_3 $.

\item There is a short exact sequence
		\[\begin{tikzcd}\label{c(cp7)}
			0 & {\Theta_{14}} & {\cc{\csum_k\Cp{7}}} & {\opl{k-1}\cc{\Cp{6}}\opl{}\kr{h_C^*}} & 0,
			\arrow[from=1-1, to=1-2]
			\arrow["{{d^*}}", from=1-2, to=1-3]
			\arrow["\iota^*", from=1-3, to=1-4]
			\arrow[from=1-4, to=1-5]
		\end{tikzcd}\]
		where $ \Theta_{14}\cong \Z_2$, $ \cc{\Cp{6}}\cong\Z_2\opl{} \Z_3$ and $ \kr{h_C^*:\cc{\Cp{6}}\lr\Th{13}}\cong\Z_2 $.

\item  There is a short exact sequence
		\[\begin{tikzcd}\label{c(cp8)}
			0 & {\Theta_{16}} & {\cc{\csum_k\Cp{8}}} & {\opl{k-1}\cc{\Cp{7}}\opl{}\kr{h_C^*} }& 0,
			\arrow[from=1-1, to=1-2]
			\arrow["{{d^*}}", from=1-2, to=1-3]
			\arrow["\iota^*", from=1-3, to=1-4]
			\arrow[from=1-4, to=1-5]
		\end{tikzcd}\]
		where $ \Theta_{16}\cong\Z_2 $, $ \cc{\Cp{7}}\cong\Z_2^2 $ and $ \kr{h_C^*:\cc{\Cp{7}}\lr\Th{15}}\cong\Z_2  $.
		
\end{enumerate}
\end{cor}

\begin{proof}  Using \Cref{ic of connected sum of projective spaces is trivial} $(i)$, for $ n\leq8 $, we conclude that the group $ I_c(\csum_k\Cp{n}) $ is trivial which implies that the induced map $ d^*:\Th{2n}\to\cc{\csum_k\Cp{n}} $ is injective. Therefore,  for $ \csum_k\Cp{n} $ the short exact sequence derived from \Cref{thm:gen short exact sequence,main thm for connected sum of same copy} for $Y=Top/O$ is as follows:
\begin{equation}\label{short ex cor:sq for cpn}
    \begin{tikzcd}
	0 & {\Th{2n}} & {\cc{\csum_k\Cp{n}}} & {\opl{k-1}\cc{\Cp{n-1}}\opl{}\kr{h_C^*}} & 0~,
	\arrow[from=1-1, to=1-2]
	\arrow["{d^*}", from=1-2, to=1-3]
	\arrow["{\iota^*}", from=1-3, to=1-4]
	\arrow[from=1-4, to=1-5]
\end{tikzcd}
\end{equation}
where $ h_C:\Sp{2n-1}\to \Cp{n-1} $ is the Hopf map. Then we have the following:
\begin{enumerate}[$(i)$]
\item  Consider $ n=4$ in  \eqref{short ex cor:sq for cpn}. It has been proved that $\cc{\Cp{3}} $ is trivial  in \cite[Theorem 2.3 (i)]{RK16ClassificationSmoothStructuresComplexProjectiveSpace}. This implies that $ d^*:\Theta_{8}\to\cc{\csum_k\Cp{4}} $ is an isomorphism.	In particular, we get that  $$ \mathcal{C}(\csum_k\Cp{4}) = \{[\csum_k\Cp{4}], [\csum_k\Cp{4} \csum \Sigma^8]\}, $$ where $ \Sigma^8 $ is an exotic $ 8 $-sphere. This conveys that the $\csum_k\Cp{4}$ has exactly two distinct smooth structures up to concordance, and at most two distinct smooth structures up to diffeomorphism.

\item It is evident from $ (i) $ that $ \cc{\Cp{4}}\cong\Z_2 $ which is also computed in \cite[Theorem 2.3 (ii)]{RK16ClassificationSmoothStructuresComplexProjectiveSpace}. Note that the group $\Th{10}=\Z_2\opl{}\Z_3$. Moreover, by \cite[Proposition 2.6]{Mosher1968} the composition 
		\begin{tikzcd}[sep=small]
			{\Sp{9}} & {\Cp{4}} & {\Sp{8}}
			\arrow["h_C", from=1-1, to=1-2]
			\arrow["d", from=1-2, to=1-3] 
		\end{tikzcd}
		 is null-homotopic. This implies that the induced map $ h_C^*:\cc{\Cp{4}}\to \Th{9}$ is trivial. Therefore, by substituting the corresponding groups for $n=5$ in the sequence \eqref{short ex cor:sq for cpn}, we get the required short exact sequence.

\item  Let $n=6$ in the sequence \eqref{short ex cor:sq for cpn}. As the group $ \Th{12} $ is trivial, we infer that the induced map $ \iota^* $ is a monomorphism, and hence an isomorphism. Further, in \cite[Theorem 2.7 (ii)]{RK-2017-CP5-8}, it is proved that $\cc{\Cp{5}}\cong\Z_2^2\oplus\Z_3$, and the image of the induced map  $ h_C^*:\cc{\Cp{5}}\to\Th{11}$ is isomorphic to $\Z_2 $. Therefore, by the isomorphism theorem of groups, we conclude that $\kr{h_C^*}$ is isomorphic to $\Z_2\oplus\Z_3$. Thus we obtain the required isomorphic.

\item Consider the sequence \eqref{short ex cor:sq for cpn} for $n=7$. In 
\cite[Theorem 2.7(iii)]{RK-2017-CP5-8}, it is proved that $\cc{\Cp{6}}\cong\Z_2\oplus\Z_3$, and the induced map $ h_C^*:\cc{\Cp{6}}\to\Th{13}$ is an epimorphism, where $\Th{13} = \Z_3$. This implies that $\kr{h_C^*}\cong\Z_2$. 

\item It has been proved in \cite[Theorem 2.9]{RK-2017-CP5-8} that the group $\cc{\Cp{7}}\cong\Z_2^2$, and the image of the induced map $ h_C^*:\cc{\Cp{7}}\to\Th{{15}} $ is isomorphic to $ \Z_2 $. Since the group $ \Th{16}=\Z_2 $, we obtain the required short exact sequence from the sequence  \eqref{short ex cor:sq for cpn} by putting $n=8$.

\end{enumerate}
\end{proof}

From \Cref{cor:sq for cpn} $ (i) $ and $ (iii) $ we have $ \cc{\csum_k\Cp{4}}\cong\Z_2 $ and $ \cc{\csum_k\Cp{6}}\cong \Z_2^{2k-1}\opl{}\Z_3^k $, respectively. 
Further, we explicitly compute the concordance structure sets of $\csum_k\Cp{5}$ and $\csum_k\Cp{7}$. The proofs exhibit that the sequences in \Cref{cor:sq for cpn} $ (ii) $ and $ (iv) $ split.

\begin{thm}
Let $ k\in \N $. Then
\begin{enumerate}[$ (i) $]
\item $ \cc{\csum_k\Cp{5}}\cong\Z_2^{k+1}\opl{}\Z_3 $

\item $ \cc{\csum_k\Cp{7}}\cong\Z_2^{k+1}\opl{}\Z_3^{k-1} $ 

\end{enumerate}
\end{thm}
\begin{proof}
%%%%%%%%%%%%%%%%%
The localization at a prime ideal is an exact functor on the category of abelian groups. This property  be extensively used to compute required groups by deducing splittings of the below sequences from \Cref{cor:sq for cpn} $(ii)$ and $(iv)$, 
	\[\begin{tikzcd}
			0 & {\Theta_{10}} & {\cc{\csum_k\Cp{5}}} & {\opl{k}\cc{\Cp{4}}} & 0
			\arrow[from=1-1, to=1-2]
			\arrow["{{d^*}}", from=1-2, to=1-3]
			\arrow["{{\iota^*}}", from=1-3, to=1-4]
			\arrow[from=1-4, to=1-5]
		\end{tikzcd}\]
and
	\[\begin{tikzcd}
			0 & {\Theta_{14}} & {\cc{\csum_k\Cp{7}}} & {\opl{k-1}\cc{\Cp{6}}\opl{}\kr{h^*}} & 0,
			\arrow[from=1-1, to=1-2]
			\arrow["{{d^*}}", from=1-2, to=1-3]
			\arrow["{\iota^*}", from=1-3, to=1-4]
			\arrow[from=1-4, to=1-5]
		\end{tikzcd} \]
respectively.
The primes occurring in the primary decomposition of groups $\cc{\csum_k\Cp{5}}$ and $\cc{\csum_k\Cp{7}}$ are $2$ and $3$. Note that the localization of these short exact sequences at $3$, results in $\cc{\csum_k\Cp{5}}_{(3)}\cong\Z_3$ and $\cc{\csum_k\Cp{7}}_{(3)}\cong\Z_3^{k-1}$ respectively. Thus the splittings of the above sequences are equivalent to the non-existence of order $4$ element in middle groups.
\begin{enumerate}[$(i)$]
    \item Let $ \alpha:\vee_k\Cp{3} \to{\csum_k\Cp{5}} $ be the inclusion map. Since $ \alpha $ is inclusion of the $6$-skeleton, we have 
        $$\faktor{\csum_k\Cp{5}}{\vee_k\Cp{3}}\simeq \left(\faktor{\vee_k\Cp{4}}{\vee_k\Cp{3}} \right) \underset{\beta\circ\xi}{\cup}\D{10}~,$$ 

 where $\xi:\Sp{9}\to \vee_k\Cp{4}$ is the attaching map to obtain $\csum_k\Cp{5}$ and $\beta:\vee_k\Cp{4}\to \sfrac{\vee_k\Cp{4}}{\vee_k\Cp{3}}$ is the quotient map. Furthermore, the space $\sfrac{\vee_k\Cp{4}}{\vee_k\Cp{3}}$ is homotopy equivalent to $\vee_k(\sfrac{\Cp{4}}{\Cp{3}})\simeq\vee_k\Sp{8}$. Hence we have $\beta\simeq \vee_kd$, where $d:\Cp{4}\to\Sp{8}$ is the degree one map. In \cite[Proposition 2.6]{Mosher1968}, it is showed that the composition $d\circ h_C:\Sp{9}\to\Cp{4} \to\Sp{8}$  is nullhomotopic, where $h_C$ is the Hopf map. Therefore, in the following homotopy commutative diagram,
\[\begin{tikzcd}
	& {\underset{k}{\vee}\Sp{9}} \\
	{\Sp{9}} && {\underset{k}{\vee}\Cp{4}} &[2em] {\underset{k}{\vee}\Sp{8}}
	\arrow["{\beta\simeq\underset{k}{\vee}d}"', from=2-3, to=2-4]
	\arrow["\xi"', from=2-1, to=2-3]
	\arrow["p",  from=2-1, to=1-2]
	\arrow["{\underset{k}{\vee}h_C}",  from=1-2, to=2-3]
\end{tikzcd}\]
 we have $\beta\circ\xi\simeq (\vee_k(d\circ h_C))\circ\, p$ to be nullhomotopic. This implies the space $\sfrac{\csum_k\Cp{5}}{\vee_k\Cp{3}}\simeq\vee_k\,\Sp{8}\vee\Sp{10}$. \\
Now consider the cofiber sequence %obtained from $ \iota $,
   \begin{tikzcd}
	{\vee_k\Cp{3}} & {\csum_k\Cp{5}} & {\vee_k\Sp{8}\vee\Sp{10}}
	\arrow["\alpha", hook, from=1-1, to=1-2]
	\arrow["q", from=1-2, to=1-3]
\end{tikzcd} and apply the functor $[\_,Top/O]$ on it. Then the induced map $q^*:\opl{k}\Th{8}\opl{}\Th{10}\to \cc{\csum_k\Cp{5}}$ is surjective because the group $\cc{\Cp{3}}$ is trivial. Finally, the groups $\Th{8_{(2)}}=\Z_2$ and $\Th{10_{(2)}}=\Z_2$ implies that $\cc{\csum_k\Cp{5}}$ has no order $4$ element. 
\item  Consider the cofiber sequence \begin{tikzcd}
	{\sqcup_{k-1}\Sp{13}} & {\csum_k\Cp{7}} & {\vee_k\Cp{7},}
	\arrow["\gamma", hook, from=1-1, to=1-2]
	\arrow["\rho", from=1-2, to=1-3]
\end{tikzcd} where the map $\gamma$ is the inclusion of $ \sqcup_{k-1}\Sp{n-1} $ on gluing parts during the connected sum operation in $\csum_k \Cp{7} $ such that the cofiber of $ \gamma $ is $ \vee_k\Cp{7} $. In the long exact sequence obtained by functor $ [\_,Top/O] $, the induced map $\rho^*_{(2)}:\opl{k}\Cp{7}_{(2)}\to \cc{\csum_k\Cp{7}}_{(2)}$ is surjective since $\Th{13_{(2)}}$ is trivial. Therefore the group $\cc{\Cp{7}}_{(2)}\cong\Z_2^2$ implies that $\cc{\csum_k\Cp{7}}$ has no order $4$ element.

\end{enumerate} 
This completes the proof. 
\end{proof}

%%%%%%%%%%%%%%%%%%%%%%%%

%%%%%%%%%%%%%%%%%%%%%%%
 Now we show analogous results hold in the case of $k$-fold connected sum of quaternionic projective spaces. Recall that $ \Hp{n}=\Hp{n-1}\cup_{h_Q}\D{4n} $, where the attaching map $ h_Q:\Sp{4n-1}\to \Hp{n-1} $ is the Hopf map. Computations of groups $\cc{\Hp{n}}$ for $n\leq 4$ are done in \cite[Theorem 3.1]{RKInertiaGroupsAndSmoothStructuresOnQuaternionicProjectiveSpaces}. %In \cite{Arvind2004}, it is proved that the concordance inertia group of $ \Hp{n} $ is trivial for $n\leq 4$. 
  Here, by virtue of \Cref{thm:gen short exact sequence,,main thm for connected sum of same copy} in case of $\cc{\csum_k\Hp{n}}$ we obtain the following:

\begin{cor}{\label{cor: quaternionic}}
\begin{enumerate}[(i)]
\item 	There is an isomorphism $$ d^*:\Theta_{8}\lr\cc{\csum_k\Hp{2}}, $$ 
		where $ \Theta_{8}\cong\Z_2 $. %\textcolor{cyan}{$ \cc{\csum_k\Hp{2}}\cong\cc{\Hp{2}} $}
		
		\item  There is an isomorphism 
		$$\iota^*:\cc{\csum_k\Hp{3}}\lr\opl{k}\cc{\Hp{2}},$$
		where $ \cc{\Hp{2}}\cong\Z_2 $. %\textcolor{cyan}{$ \cc{\csum_k\Hp{3}}\cong\opl{k}\cc{\Hp{3}} $}
		
		\item There is a short exact sequence
		\[\begin{tikzcd}
			0 & {\Theta_{16}} & {\cc{\csum_k\Hp{4}}} & {\opl{k-1}\cc{\Hp{3}}} & 0,
			\arrow[from=1-1, to=1-2]
			\arrow["{{d^*}}", from=1-2, to=1-3]
			\arrow["{{\iota^*}}", from=1-3, to=1-4]
			\arrow["0", from=1-4, to=1-5]
		\end{tikzcd}\] 
		where $ \Theta_{16}\cong\Z_2 $ and $ \cc{\Hp{3}}\cong\Z_2 $.
\end{enumerate}
\end{cor}
\begin{proof} It follows from the \Cref{ih of connected sum} and \Cref{ic of connected sum of projective spaces is trivial} that the induced map $ d^*:\Th{4n}\to\cc{\csum_k\Hp{n}} $ is injective for $ n\leq4 $. Therefore, using \Cref{thm:gen short exact sequence,,main thm for connected sum of same copy} for $ \csum_k\Hp{n} $ and $Y=Top/O$ we obtain the following short exact sequence
\begin{equation}\label{short ex sq for hpn}
    \begin{tikzcd}
	0 & {\Th{4n}} & {\cc{\csum_k\Hp{n}}} & {\opl{k-1}\cc{\Hp{n-1}}\opl{}\kr{h_Q^*}} & 0
	\arrow[from=1-1, to=1-2]
	\arrow["{d^*}", from=1-2, to=1-3]
	\arrow["{\iota^*}", from=1-3, to=1-4]
	\arrow[from=1-4, to=1-5]
\end{tikzcd}
\end{equation}
where $ h_Q:\Sp{4n-1}\to \Hp{n-1} $ is the Hopf map. Then we have the following:%%%%%%%%%%%%%%
\begin{enumerate}[$(i)$]
\item The group $ \cc{\Hp{1}} $ is trivial because $ \Hp{1} $ is homeomorphic to $ \Sp{4} $ and $ \cc{\Sp{4}} $ is trivial. Thus for $n=2$, the induced map $ d^* $ in the sequence \eqref{short ex sq for hpn} is an isomorphism.

\item Let $n=3$ in the sequence \eqref{short ex sq for hpn}. The group $\Th{12}$ is trivial, which results in an isomorphism of the map $\iota^*$. Next, it follows from $(i)$ that the group $\cc{\Hp{2}}\cong\Z_2$, and by \cite[Theorem 3.1 (i)]{RKInertiaGroupsAndSmoothStructuresOnQuaternionicProjectiveSpaces} the induced map $ h_Q^*:\cc{\Hp{2}}\to \Th{11} $ is trivial. This completes the proof of $(ii)$.  

\item It is evident from  $(ii)$ that the group $\cc{\Hp{3}}\cong\Z_2$. In \cite[Theorem 3.1 (ii)]{RKInertiaGroupsAndSmoothStructuresOnQuaternionicProjectiveSpaces}, it is proved that the induced map $ h_Q^*:\cc{\Hp{3}}\to\Th{15}$ is a monomorphism. Since the group $ \Th{16}=\Z_2 $, we obtain the required short exact sequence from the sequence  \eqref{short ex sq for hpn} by substituting $n=4$.
\end{enumerate}
\end{proof}
%%%%%%%%%%%%%%%%%%%%%%%%%%%%%%%%
%%%%%%%%%%%%%%%%%%%%%%%%%%%%%
 The existence of a short exact sequence in the corollary below is similar to \Cref{cor: quaternionic} for the connected sum of octonionic projective space $\Op{2}$. 
\begin{cor}
	There is a short exact sequence 
	\[\begin{tikzcd}
		0 & {\Theta_{16}} & {\cc{\csum_k\Op{2}}} & {\opl{k}\Theta_{8}} & 0,
		\arrow[from=1-1, to=1-2]
		\arrow["{{d^*}}", from=1-2, to=1-3]
		\arrow["{{\iota^*}}", from=1-3, to=1-4]
		\arrow[ from=1-4, to=1-5]
	\end{tikzcd}\]
	where $\Theta_{16}\cong \Z_2 $ and $ \Theta_{8}\cong\Z_2 $.
\end{cor}
\begin{proof} 
Recall that $ \Op{2}=\Op{1}\cup_{h_O}\D{16} $, where the attaching map $ h_O:\Sp{15}\to \Op{1} $ is the Hopf map. In \cite[Theorem 2.3 (iii)]{RKInertiaGroupsAndSmoothStructuresOf2nmfld}, it is proved that the concordance inertia group of $ \Op{2} $ is trivial. By \Cref{ic of connected sum is sum of I_c}, we conclude that for all $ k\in \N $, the concordance inertia group of $ \csum_k\Op{2} $ is trivial. Hence the induced map $d^*:\Th{16}\to\cc{\csum_k\Op{2}}$ is injective. Hence using \Cref{thm:gen short exact sequence,main thm for connected sum of same copy} for $ \cc{\csum_k\Op{2}} $, we obtain the following short exact sequence:
\[\begin{tikzcd}
	0 & {\Th{16}} & {\cc{\csum_k\Op{2}}} & {\opl{k-1}\cc{\Op{1}}\opl{}\kr{h_O^*}} & 0
	\arrow[from=1-1, to=1-2]
	\arrow["{d^*}", from=1-2, to=1-3]
	\arrow["{\iota^*}", from=1-3, to=1-4]
	\arrow[from=1-4, to=1-5]
\end{tikzcd}\]

  Moreover, by \cite[Theorem 4.5]{RKHomotopyInertiaGroupsAndTangentialStructures}, the induced map $ h^*_O:\Th{8}\to\Th{15} $, maps the non-trivial element of $ \Th{8} $ to a null homotopic map $ f:\Sp{15}\to\Sp{8}\to Top/O $. As a result, we obtain the required short exact sequence.
\end{proof}
%%%%%%%%%%%%%%%%%%%%%
%%%%%%%%%%%%%%%%%%%%%%%%%%%%%
 Similarly using \cite[Theorem 3.1 (i)]{RKInertiaGroupsAndSmoothStructuresOnQuaternionicProjectiveSpaces} and \cite[Theorem 2.7]{RK-2017-CP5-8}, we have the concluding result.
\begin{cor}
$ \cc{(\csum_k\Cp{6})\csum(\csum_l\Hp{3})}\cong\Z_2^{2k+l-1}\opl{}\Z_3^k $
\end{cor}
Observe that, this result also follows from \Cref{cor: C of n-1 connected 2n mfld}.
% \begin{proof} Consider the short exact sequence in the \Cref{thm:gen short exact sequence} for $ (\csum_k\Cp{6})\csum(\csum_l\Hp{3}) $ with $ \xi=(\vee_kh_C\vee_lh_Q)\circ p $, where $ h_C $ and $ h_Q $ are corresponding Hopf maps and $ p $ is the pinch map. As $ \Th{12}=0 $ implies $ \iota^* $ is injective, and hence an isomorphism. To complete the proof, we need to compute $ \kr{\xi^*} $. From equation \eqref{eq:ker of xi is ker of sum of hi} $, \kr{\xi^*}=\kr{h_C^*+h_Q^*} $. By \cite[Theorem 3.1 (i)]{RKInertiaGroupsAndSmoothStructuresOnQuaternionicProjectiveSpaces} $ h_Q^* $ is trivial, which implies $ \kr{\xi^*}=\kr{h_C^*}\opl{l}\Hp{2} $. Finally, by \cite[Theorem 2.7 (ii)]{RK-2017-CP5-8} the $ \kr{h_C^*}\cong\Z_2\opl{}\Z_3 $ completes the proof.
% \end{proof}

\vspace{1cm}
\noindent\textbf{Acknowledgement: }
The author would like to express gratitude to Prof. R. Kasilingam for valuable guidance and support throughout the development of this paper. Additionally, the author acknowledges the University Grants Commission for their support through the Senior Research Fellowship.

%%%%%%%%%%%%%%%%%%%%%%%%%%%%%%%%%%%%%%%%%%%%%%%%%%%%%%%%%%%%%%%%%%%%%%%%%%%%%%%%%%%%%%%%%%%%%%%%%%%%%%%%%%%%
\end{document}